\documentclass{article}

\usepackage{authblk}
\title{Restrictions of Infinite Circuits}
\author{Evan Leach}
\affil{Department of Mathematics, UCLA}
\date{September 12, 2026}

\usepackage[letterpaper, total={5.5in, 8in}]{geometry}
\usepackage{setspace}
\usepackage{amsmath}
\usepackage{amssymb}
\usepackage{amsthm}
\usepackage{mathrsfs}
\usepackage{enumitem}
\usepackage{forest}
\usepackage{bm}
\usepackage{hyperref}
\hypersetup{
    colorlinks=true,
    linkcolor=blue,
    filecolor=blue,      
    urlcolor=blue,
    citecolor=magenta,
    }

\usepackage[
    backend=biber,
    style=alphabetic,
    maxalphanames=3,
    maxbibnames=99
]{biblatex}

\newcommand{\NN}{\mathbb{N}}

\newcommand{\AND}{\wedge}
\newcommand{\OR}{\vee}
\newcommand{\NOT}{\neg}
\newcommand{\BSigma}{\boldsymbol{\Sigma}}
\newcommand{\BPi}{\boldsymbol{\Pi}}
\newcommand{\BDelta}{\boldsymbol{\Delta}}
\newcommand{\AC}{\mathsf{AC}^0}
\newcommand{\ACd}{\mathsf{AC}^0_d}
\newcommand{\ACdm}{\mathsf{AC}^0_{d-1}}

\DeclareMathOperator{\rank}{rank}

\theoremstyle{plain}
\newtheorem{theorem}{Theorem}[section]

\newtheorem{lemma}[theorem]{Lemma}
\newtheorem{corollary}[theorem]{Corollary}

\theoremstyle{definition}
\newtheorem{definition}[theorem]{Definition}
\newtheorem{proposition}[theorem]{Proposition}
\newtheorem{remark}[theorem]{Remark}
\newtheorem{conjecture}[theorem]{Conjecture}

\begin{document}

\null
\nointerlineskip
\vfill
\let\snewpage \newpage
\let\newpage \relax
\maketitle

\vspace{2cm}

\begin{abstract}

We present a new approach to lower bounds on Borel rank via a connection between Borel sets and infinite circuits. We prove an infinite analog to Håstad's switching lemma called the \emph{restriction lemma}, which shows that by fixing the values of some inputs, we can simultaneously reduce the complexity of one Borel function while largely maintaining the complexity of another. This result provides a purely combinatorial proof of the Borel hierarchy theorem, as well as simple proofs of various Ramsey-like properties of Borel sets and functions. We prove that the restriction lemma is sharp and also discuss counterexamples demonstrating the limitations of the restriction technique, in the context of both Borel functions and $\AC$ circuit families.

\end{abstract}

\vspace{2cm}

\let \newpage \snewpage
\vfill 

\newpage

\tableofcontents

\newpage

\section{Introduction}

Proving lower bounds on Borel rank relies on the Borel hierarchy theorem and its proof via universal sets and diagonalization. In this paper, we give a more concrete alternative approach based on a correspondence between Borel sets and infinite circuits. Along with proving the Borel hierarchy theorem, our techniques give simple proofs of several Ramsey-like properties of Borel sets and functions.

An infinite circuit is a countable, directed, well-founded graph with vertices labeled as $\AND$-gates, $\OR$-gates, literals $x_i$, and negated literals $\overline{x}_i$. These circuits compute maps which take a string $x \in 2^\NN$ as input and output the evaluation of some designated \emph{output gate} to $0$ or $1$. The functions $f \colon 2^\NN \rightarrow \{0,1\}$ computed by infinite circuits are precisely the characteristic functions of Borel sets $B \subseteq 2^\NN$ (see Section~\ref{sec:correspondence} for a proof). Moreover, the Borel rank of $B$ is closely related to the depth of the corresponding circuit.

Many classical theorems about Borel sets have analogous results for $\AC$, the class of functions $f \colon 2^{<\NN} \rightarrow \{0,1\}$ computable by polynomial-sized circuit families with bounded depth. The Borel hierarchy theorem corresponds to Sipser's result \cite{sipser1983} that, for any $d \in \NN$, there are functions computed by polynomial-sized circuits of depth $d+1$ but not by polynomial-sized circuits of depth $d$. The fact that no infinite parity function $f \colon 2^\NN \rightarrow \{0,1\}$ (a function whose output flips whenever a single input bit is flipped) is Borel corresponds to Furst, Saxe, and Sipser's result \cite{parity1981} that parity is not in $\AC$. Both of these theorems about finite circuits were originally proven using inspiration from the analogous infinitary result \cite{parity1981, sipser1983, sipser1984}.

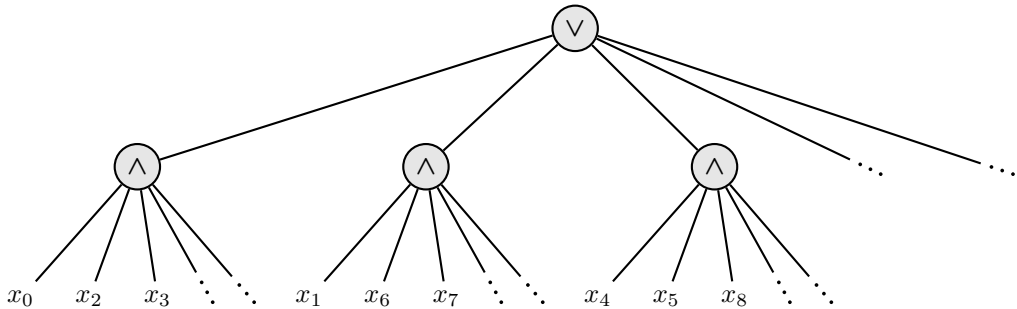
\begin{figure}[b]
    \centering
    \begin{forest}
      for tree={
        math content,
        s sep=3mm, 
        l sep=12mm, 
        tier/.option=level,
        edge={thick}
      },
      gate/.style={circle, draw, thick, fill=gray!20, inner sep=0pt, minimum size=6mm},
      ellipsis child/.style={
        edge path={
          \noexpand\draw[\forestoption{edge}] (!u) -- ($()!4.5mm!(!u)$);
          \noexpand\fill ($()!3mm!(!u)$) circle (0.7pt);
          \noexpand\fill ($()!1.5mm!(!u)$) circle (0.7pt);
          \noexpand\fill () circle (0.7pt);
        }
      }
      [\bm{\OR}, gate
        [\bm{\AND}, gate
          [x_0]
          [x_2]
          [x_3]
          [, ellipsis child]
          [, ellipsis child]
        ]
        [\bm{\AND}, gate
          [x_1]
          [x_6]
          [x_7]
          [, ellipsis child]
          [, ellipsis child]
        ]
        [\bm{\AND}, gate
          [x_4]
          [x_5]
          [x_8]
          [, ellipsis child]
          [, ellipsis child]
        ]
        [, coordinate, no edge]
        [, coordinate, no edge]
        [, coordinate, no edge]
        [, coordinate, no edge]
        [, coordinate, no edge]
        [, ellipsis child]
        [, coordinate, no edge]
        [, coordinate, no edge]
        [, coordinate, no edge]
        [, coordinate, no edge]
        [, ellipsis child]
      ]
    \end{forest}
    \caption{A circuit computing a $2$-tree function. Each gate has countably infinite fan-in.}
    \label{fig:alternating_tree}
\end{figure}

The sharpest versions of the $\AC$ depth hierarchy theorem and parity lower bounds were proven by Håstad in \cite{hastad} using a theorem called the \emph{switching lemma}, a cornerstone result in circuit complexity \cite{arora2009, brustmann1987, servedio2022, lyu2022}. The switching lemma essentially states that, by fixing a reasonably small number of input bits to $0$ or $1$, one can dramatically simplify a circuit. Such a partial setting of the input bits is called a \emph{restriction}, and we write $f^\rho$ to denote the application of a restriction $\rho$ to a function $f$.

In this paper, we prove an infinite analog of the switching lemma, which we call the \emph{restriction lemma}. Our result concerns the \emph{$\gamma$-tree functions}, which are computed by circuits obtained from labeling the vertices of a fully-branching tree of rank $\gamma < \omega_1$ with alternating $\AND$- and $\OR$-gates (see Section~\ref{sec:alternating-trees} for a complete definition and Figure~\ref{fig:alternating_tree} for an example). These functions compute canonical $\BSigma^0_\gamma$- and $\BPi^0_\gamma$-complete sets. The restriction lemma gives us a way to simultaneously reduce the complexity of a $\BSigma^0_{\alpha+1}$-measurable function $f \colon 2^\NN \rightarrow 2^\NN$ and an $(\alpha+\beta)$-tree function $g \colon 2^\NN \rightarrow \{0,1\}$ by applying a restriction. The precise statement of the theorem is as follows:

\begin{theorem}[The restriction lemma]
\label{restriction}

For any ordinals $\alpha, \beta < \omega_1$, $\BSigma^0_{\alpha+1}$-measurable function $f \colon 2^\NN \rightarrow 2^\NN$, and $(\alpha+\beta)$-tree function $g \colon 2^\NN \rightarrow \{0,1\}$, there exists a restriction $\rho$ such that $f^\rho$ is continuous and $g^\rho$ is a $\beta$-tree function.

\end{theorem}

When $\beta \geq 1$ and the codomain of $f$ is $\{0,1\}$ instead of $2^\NN$, we can choose $\rho$ such that $f^\rho$ is constant, rather than merely continuous:

\begin{corollary}
\label{constant-restriction}

For any ordinals $\alpha, \beta < \omega_1$ with $\beta \geq 1$, $\BSigma^0_{\alpha+1}$-measurable function $f \colon 2^\NN \rightarrow \{0,1\}$, and $(\alpha + \beta)$-tree function $g \colon 2^\NN \rightarrow \{0,1\}$, there exists a restriction $\rho$ such that $f^\rho$ is constant and $g^\rho$ is a $\beta$-tree function.

\end{corollary}

While Sipser proved a version of these results for finite-depth circuits \cite{sipser1983}, his techniques fundamentally break down at the transfinite level. Roughly speaking, he provides a construction which reduces the rank of $f$ and $g$ by $1$, but iterating his argument $\omega$-many times causes the complexity of $g$ to completely collapse. Our proof circumvents these difficulties using a new inductive hypothesis which reduces the complexity of $f$ and $g$ by a possibly infinite amount. Our argument also requires a new technique for finding restrictions at limit stages of the induction. As was the case for Håstad's switching lemma in finite circuit complexity, the restriction lemma establishes a variety of results in descriptive set theory, including the Borel hierarchy theorem and the fact that no infinite parity function is Borel.

Before the modern proof of the Borel hierarchy theorem via universal sets and diagonalization, Baire established separations between the first several levels of the Borel hierarchy using concrete arguments involving Baire classes of functions \cite{baire1899, baire1906}. However, separating higher levels of the hierarchy required increasingly complex arguments, and it took until 1983 for Sipser to give a purely combinatorial proof that $\BSigma^0_n \neq \BPi^0_n$ for all $1 \leq n < \omega$ \cite{sipser1983}. The restriction lemma provides the first purely combinatorial proof of the full Borel hierarchy theorem by establishing that $\BSigma^0_\alpha \neq \BPi^0_\alpha$ for all $1 \leq \alpha < \omega_1$ without universal sets or diagonalization.

The restriction lemma lies at the intersection of circuit complexity (as an infinite version of Håstad's switching lemma, see also Section~\ref{sec:parity} and Section~\ref{sec:ac-limitations}), infinite dimensional Ramsey theory (see Theorem~\ref{continuous-subcube} and Theorem~\ref{constant-subcube}), and the French school of descriptive set theory analyzing the relationship between various levels of the Borel hierarchy (we will see in Section~\ref{sec:sharpness} that the restriction lemma proves results similar to the corollaries of Louveau and Saint Raymond's theorem in \cite{louveau1987}).

\subsection{Outline of the paper}

We begin with a review of Borel sets and functions in Section~\ref{sec:preliminaries}, and we define the notion of an infinite circuit in Section~\ref{sec:infinite-circuits}. We establish the key connection between Borel sets and circuits in Section~\ref{sec:correspondence}. We define tree functions in Section~\ref{sec:alternating-trees} and restrictions (along with some important related terminology) in Section~\ref{sec:restrictions}. Section~\ref{sec:restriction-lemma} features the proofs of Theorem~\ref{restriction} and Corollary~\ref{constant-restriction}.

We explore the consequences of the restriction lemma in Section~\ref{sec:consequences}, beginning with a new proof of the Borel hierarchy theorem in Section~\ref{sec:hierarchy}. We prove in Section~\ref{sec:constant} that any Borel function $f \colon 2^\NN \rightarrow 2^\NN$ can be made continuous by applying a restriction preserving infinitely many variables, and we use this to show that no infinite parity function is Borel in Section~\ref{sec:parity}. We also give a generalization of the restriction lemma to arbitrary Borel functions in Section~\ref{sec:restriction-for-non-trees}, rather than just $\gamma$-tree functions. This requires us to broaden the class of restrictions we allow, and we show that the theorem fails without this relaxation. We also discuss similar limitations in applying restrictions to functions in $\AC$.

We prove the sharpness of the restriction lemma in its strongest form (Theorem~\ref{restriction}) in Section~\ref{sec:sharpness}, and we discuss how the weaker form (Corollary~\ref{constant-restriction}) is not sharp. We introduce an open conjecture regarding how far Corollary~\ref{constant-restriction} is from optimal, and we provide some partial progress towards a resolution.

\section{Borel sets and infinite circuits}
\label{sec:borel-and-circuits}

Define $2^{<\NN} = \bigcup_{n \in \NN} \{0,1\}^n$, and let $w_0w_1 \cdots w_{k-1}w_k$ denote the concatenation of strings $w_0,w_1, \dots, w_{k-1} \in 2^{<\NN}$ and $w_k \in 2^{<\NN} \cup 2^\NN$. We consider the \emph{Cantor space} $2^\NN$, equipped with the topology generated by basic open sets of the form $w2^\NN$ for $w \in 2^{<\NN}$.

\subsection{Preliminaries}
\label{sec:preliminaries}

We first recall some standard definitions from descriptive set theory about Borel sets and functions. See \cite{kechris1995} for a detailed treatment.

\begin{definition}

A set $S \subseteq 2^\NN$ is called $\BSigma^0_0$ or $\BPi^0_0$ if it is clopen. Given $1 \leq \alpha < \omega_1$, we call a set $\BSigma^0_\alpha$ (resp.\ $\BPi^0_\alpha$) if it is a countable union (resp.\ intersection) of sets $S_n$ for $n \in \NN$, each of which is in $\BPi^0_{\beta_n}$ (resp.\ $\BSigma^0_{\beta_n}$) for some $\beta_n < \alpha$. We call a set $\BDelta^0_\alpha$ if it is simultaneously $\BSigma^0_\alpha$ and $\BPi^0_\alpha$.

\end{definition}

\begin{remark}

This definition agrees with the usual definition of the Borel hierarchy since the Cantor space is zero-dimensional (has a clopen basis). In particular, the open (resp.\ closed) sets of $2^\NN$ are precisely the countable unions (resp.\ intersections) of clopen sets.

\end{remark}

\begin{definition}

A set $S \subseteq 2^\NN$ is \emph{Borel} if it is $\BSigma^0_\alpha$ for some $\alpha < \omega_1$. We let $\mathcal{B}$ denote the class of Borel sets. The \emph{rank} of a Borel set $B$ is the smallest $\alpha$ such that $B \in \BSigma^0_\alpha \cup \BPi^0_\alpha$.

\end{definition}

\begin{definition}

Given a class $\mathcal{C} \subseteq \mathcal{P}(2^\NN)$, we call a function $f \colon 2^\NN \rightarrow Y$ (where $Y = 2^\NN$ or $\{0,1\}$) \emph{$\mathcal{C}$-measurable} if $f^{-1}(S) \in \mathcal{C}$ whenever $S \subseteq Y$ is open.

\end{definition}

\begin{definition}

A function $f \colon 2^\NN \rightarrow Y$ is \emph{Borel} if it is $\mathcal{B}$-measurable.

\end{definition}

\begin{remark}

The facts that $2^\NN$ has a countable basis and that $\BSigma^0_\alpha$ is closed under countable unions can be used to show that every Borel function is $\BSigma^0_\alpha$-measurable for some $\alpha < \omega_1$.

\end{remark}

\begin{remark}

Suppose $B$ is a Borel set with rank $\alpha$. Then the characteristic function of $B$ is $\BSigma^0_{\alpha+1}$-measurable, and it is $\BSigma^0_\alpha$-measurable if and only if $B \in \BDelta^0_\alpha$.

\end{remark}

\begin{definition}

Given a class $\mathcal{C}$, we call a set $S \subseteq 2^\NN$ \emph{$\mathcal{C}$-hard} if for every $S' \in \mathcal{C}$, there exists a continuous map $\varphi \colon 2^\NN \rightarrow 2^\NN$ such that $x \in S' \iff \varphi(x) \in S$. We call $\varphi$ a \emph{continuous reduction.} We call a $\mathcal{C}$-hard set $S \subseteq 2^\NN$ \emph{$\mathcal{C}$-complete} if $S \in \mathcal{C}$.

\end{definition}

\subsection{Infinite circuits}
\label{sec:infinite-circuits}

\begin{definition}

A \emph{circuit} $C$ on variables $x_0, x_1, \dots$ is a countable, labeled, and directed graph satisfying the following conditions:
\begin{itemize}
    \item We call the vertices of $C$ \emph{gates}, and if there is an edge from a gate $G'$ to a gate $G$, we say that $G'$ is a \emph{child} of $G$. A circuit must be well-founded under the child relation.
    \item Any gate with fan-in (i.e.\ in-degree) at least $1$ is called an \emph{internal gate} and must be labeled with either the symbol $\AND$ or $\OR$.
    \item The remaining nodes are called \emph{input gates} and must be labeled with either $0$, $1$, or the symbol $x_i$ or $\overline{x}_i$ for some $i \in \NN$ (denoting either the $i$-th bit of $x$ or the negation of this bit). We call the labels $0,1,x_i,\overline{x}_i$ \emph{literals}.
    \item One vertex $G_{\text{out}}$ of $C$ must be designated as the \emph{output gate}.
\end{itemize}

\end{definition}

\begin{definition}

Given a gate $G$ in $C$ and some input $x \in 2^\NN$, we define the \emph{evaluation} $G(x)$ inductively (which we may do by well-foundedness). An input gate evaluates to its label, an $\AND$-gate evaluates to the conjunction of the evaluations of its children, and an $\OR$-gate evaluates to the disjunction of the evaluations of its children. We write $C(x) = G_\text{out}(x)$, so that $C$ computes a function from $2^\NN$ to $\{0,1\}$. We call two circuits \emph{equivalent} if they compute the same function.

\end{definition}

\begin{definition}

We say that a function $f \colon 2^\NN \rightarrow \{0,1\}$ \emph{computes} the set $f^{-1}(\{1\})$.

\end{definition}

\begin{remark}

Allowing for $\NOT$-gates would not increase the power of these circuits, since De Morgan's law allows us to propagate all negations to the bottom layer of the circuit.

\end{remark}

\begin{definition}

We call a circuit $C$ \emph{alternating} if no $\OR$-gate in $C$ has another $\OR$-gate as a child and no $\AND$-gate in $C$ has another $\AND$-gate as a child.

\end{definition}

\begin{remark}
\label{circuit-simplification}

By merging non-alternating gates, we see that every circuit is equivalent to an alternating circuit.

\end{remark}

\begin{definition}

Let $G$ be a gate of a circuit $C$. The \emph{subcircuit of $C$ rooted at $G$} is the subgraph of $C$ induced by $G$ and its descendants (where the descendant relation is the transitive closure of the child relation), with $G$ designated as the output gate.

\end{definition}

\begin{definition}

A \emph{child} of a circuit $C$ is a subcircuit of $C$ rooted at a child of the output gate of $C$.

\end{definition}

\begin{definition}

We call an alternating circuit $\BSigma_0$ or $\BPi_0$ if it is finite. We call an alternating circuit $C$ a \emph{$\BSigma_\alpha$-circuit (resp.\ $\BPi_\alpha$-circuit)} if its output gate is an $\OR$ (resp.\ $\AND$)-gate and, for each child $C'$ of $C$, there exists $\beta < \alpha$ such that $C'$ is a $\BPi_\beta$ (resp.\ $\BSigma_\beta$)-circuit.

\end{definition}

\begin{definition}

The \emph{rank} of an alternating circuit $C$ is the least ordinal $\alpha$ such that $C$ is a $\BSigma_\alpha$- or $\BPi_\alpha$-circuit.

\end{definition}

\begin{remark}

Remark \ref{circuit-simplification} tells us that every circuit is equivalent to a rank-$\alpha$ circuit for some $\alpha < \omega_1$. By inductively duplicating shared subcircuits, we see that every rank-$\alpha$ circuit is equivalent to a rank-$\alpha$ circuit whose underlying graph is a tree.

\end{remark}

\subsection{Equivalence of Borel sets and circuits}
\label{sec:correspondence}

Our similar choice of notation between Borel sets and circuits is deliberate, as these notions correspond precisely. The following result and its corollaries allow us to describe Borel subsets of $2^\NN$ using circuits.

\begin{lemma}
\label{clopen}

A set $S \subseteq 2^\NN$ is clopen if and only if it is computed by a finite circuit.

\end{lemma}

\begin{proof}

Any open set $S \subseteq 2^\NN$ can be written as a union of basic open sets, and if $S$ is also closed, this union can be taken to be finite by compactness. Each basic open set in such a union is computed by a finite circuit, so the forward direction follows by taking a finite disjunction of such circuits.

Conversely, any circuit consisting of a single input gate computes either $\emptyset$, $2^\NN$, or a set of the form $\{x \in 2^\NN : x_i = b\}$ for some $i \in \NN$ and $b \in \{0,1\}$. Note that these sets are all clopen. As the clopen sets are closed under finite unions and intersections, this means that any set $S$ computed by a finite circuit is also clopen.
\end{proof}

\begin{corollary}
\label{rank-equivalence}

For $\alpha  < \omega_1$, a set $S \subseteq 2^\NN$ is $\BSigma^0_\alpha$ (resp.\ $\BPi^0_\alpha$) if and only if it is computed by a $\BSigma_\alpha$ (resp.\ $\BPi_\alpha$)-circuit.

\end{corollary}

\begin{corollary}

A subset of $2^\NN$ is Borel if and only if it is computed by a circuit.

\end{corollary}

We can also characterize Borel functions using circuits:

\begin{lemma}
\label{measurable-circuits}

A function $f \colon 2^\NN \rightarrow 2^\NN$ is $\BSigma^0_\alpha$-measurable if and only if the map $x \mapsto f(x)_i$ is computed by both a $\BSigma_\alpha$-circuit and $\BPi_\alpha$-circuit for every $i \in \NN$.

\end{lemma}

\begin{proof}

Fix $\alpha < \omega_1$, and suppose first that $f \colon 2^\NN \rightarrow 2^\NN$ is $\BSigma^0_\alpha$-measurable. Then the sets $f^{-1}(\{x \in 2^\NN : x_i = b\})$ are $\BSigma^0_\alpha$ for every $i \in \NN$ and $b \in \{0,1\}$, since the sets $\{x \in 2^\NN : x_i = b\}$ are all open. Corollary~\ref{rank-equivalence} then proves the forward direction.

Conversely, suppose each map $x \mapsto f(x)_i$ for $i \in \NN$ is computed by both a $\BSigma_\alpha$-circuit and $\BPi_\alpha$-circuit. Then the sets $f^{-1}(\{x \in 2^\NN : x_i = b\})$ lie in $\BSigma^0_\alpha$ for every $i \in \NN$ and $b \in \{0,1\}$. Every open set can be generated by basic open sets of the form $\{x \in 2^\NN : x_i = b\}$ via countable unions and finite intersections, and as $\BSigma^0_\alpha$ is closed under countable unions and finite intersections, this proves the claim.
\end{proof}

Lemma~\ref{clopen} and Lemma~\ref{measurable-circuits} give us the following corollary:

\begin{corollary}
\label{continuous-circuits}

A function $f \colon 2^\NN \rightarrow 2^\NN$ is continuous if and only if the map $x \mapsto f(x)_i$ is computed by a finite circuit for every $i \in \NN$.

\end{corollary}

\subsection{Tree functions}
\label{sec:alternating-trees}

The \emph{tree functions} are a special family of Borel functions from $2^\NN$ to $\{0,1\}$ which compute $\BSigma^0_\alpha$- or $\BPi^0_\alpha$-complete sets. We will use these functions in the formulation of the restriction lemma, and they are defined in terms of circuits. We begin by defining the underlying tree structure of these circuits.

\begin{definition}

The \emph{full trees} of rank $\alpha < \omega_1$, denoted $T_\alpha$, are the rooted trees defined recursively as follows:
\begin{itemize}
    \item The tree $T_0$ consists of a single vertex, which is the root node.
    \item For any successor ordinal $\alpha + 1$, the tree $T_{\alpha+1}$ consists of $\omega$-many disjoint copies of $T_\alpha$, all of whose root nodes are connected to a single additional vertex. This additional vertex is the root node of $T_{\alpha+1}$.
    \item If $\alpha$ is a limit ordinal, the tree $T_\alpha$ consists of $\omega$-many disjoint copies of $T_\beta$ for each $\beta < \alpha$, all of whose root nodes are connected to a single additional vertex. This additional vertex is the root node of $T_\alpha$.
\end{itemize}

\end{definition}

\begin{definition}

Given a tree $T$, we say that an alternating circuit $C$ is \emph{obtained} from $T$ if the underlying graph of $C$ is $T$ (with edges directed towards the root node), the input gates of $C$ are injectively labeled by positive literals $x_i$, and the output gate of $C$ is the root node of $T$. A \emph{full alternating tree} of rank $\alpha$ is an alternating circuit obtained from $T_\alpha$.

\end{definition}

\begin{remark}

Every full alternating tree of rank $\alpha$ is either a $\BSigma_\alpha$- or $\BPi_\alpha$-circuit, depending on whether the output gate is an $\OR$-gate or an $\AND$-gate.

\end{remark}

\begin{definition}

A function $g \colon 2^\NN \rightarrow \{0,1\}$ is called an \emph{$\alpha$-tree function} if it is computed by a full alternating tree of rank $\alpha$.

\end{definition}

\begin{remark}

One can construct a reduction from any given Borel set to the appropriate full alternating tree to see that every $\alpha$-tree function defines either a $\BSigma^0_\alpha$- or $\BPi^0_\alpha$-complete set, respectively.

\end{remark}

\begin{remark}

One can modify the definition of full trees at limit ordinals $\alpha$ to merely contain a sequence of disjoint full trees $T_{\beta_0}, T_{\beta_1},\dots$ satisfying $\beta_n < \alpha$ for each $n \in \NN$ and $\sup_{n \in \NN} \beta_n = \alpha$. The results of our paper remain true, roughly because these trees all embed into each other. Our particular definition is chosen to make the proof of Lemma~\ref{inductive-restriction} easier.

\end{remark}

\subsection{Restrictions}
\label{sec:restrictions}

We now define the key concept behind the restriction lemma:

\begin{definition}

A \emph{restriction} on $2^\NN$ is a function $\rho \colon \NN \rightarrow \{0, 1, *\}$. We call variables in the set $\left\{ x_i : \rho(i) = * \right\}$ \emph{preserved} and the remaining variables \emph{fixed}.

\end{definition}

\begin{definition}

We say that $\rho'$ \emph{refines} $\rho$ and write $\rho \preceq \rho'$ if $\rho'(n) = \rho(n)$ whenever $\rho(n) \in \{0,1\}$.

\end{definition}

\begin{definition}

Given a function $f \colon 2^\NN \rightarrow Y$ (where $Y$ is equal to either $2^\NN$ or $\{0,1\}$) and restriction $\rho$, we define the \emph{restricted function} $f^\rho \colon 2^{\rho^{-1}(\{*\})} \rightarrow Y$ which takes values for the preserved variables as input and evaluates $f$ on the resulting setting.

\end{definition}

\begin{definition}

Given a circuit $C$ and restriction $\rho$, we define the \emph{restricted circuit} $C^\rho$ by replacing the labels $x_i$ with $\rho(i)$ and $\overline{x}_i$ with $\overline{\rho(i)}$ whenever $x_i$ is fixed. 

\end{definition}

\begin{definition}

We call a circuit $C^\rho$ \emph{forced} by a restriction $\rho$ if it is equivalent to a constant $0$- or $1$-circuit. We call $C^\rho$ \emph{preserved} if it computes a $0$-tree function.

\end{definition}

\begin{remark}

Being preserved is a much stronger condition than simply being not forced. A preserved circuit simply propagates the value of an input variable, which will allow us to construct tree functions from such circuits via countable conjunctions and disjunctions.

\end{remark}

\begin{definition}

Suppose a circuit $C$ is forced by a restriction $\rho$ to $0$ (resp.\ $1$). If the output gate of $C$ is an $\AND$-gate, we call $C^\rho$ \emph{suppressed} (resp.\ \emph{overriding}). If the output gate of $C$ is an $\OR$-gate, we call $C^\rho$ \emph{overriding} (resp.\ \emph{suppressed}). If $C$ is a proper subcircuit of an alternating tree $T$ and has a literal $x_i$ as its output gate, we call $C^\rho$ \emph{suppressed} (resp.\ \emph{overriding}) if the parent of $x_i$ in $T$ is an $\OR$-gate, and we call $C^\rho$ \emph{overriding} (resp.\ \emph{suppressed}) if the parent of $x_i$ in $T$ is an $\AND$-gate.

\end{definition}

\begin{remark}

In an alternating circuit, suppressed children do not affect the value of their parents and overriding children force their parents. Notice also that any full alternating rank-$1$ tree can be suppressed by assigning a single literal and can only be made overriding if all literals are assigned. Any alternating circuit obtained from a tree can be preserved by assigning an appropriate input variable to $*$ and the rest to $0$ or $1$.

\end{remark}

\section{Proof of the restriction lemma}
\label{sec:restriction-lemma}

Our goal in this section is to prove Theorem~\ref{restriction} and Corollary~\ref{constant-restriction}, and this requires us to view $f$ through the lens of circuits. In particular, Lemma~\ref{measurable-circuits} tells us that a $\BSigma^0_{\alpha+1}$-measurable function $f \colon 2^\NN \rightarrow 2^\NN$ corresponds to a countable collection of sets computable both by $\BSigma_{\alpha+1}$- and $\BPi_{\alpha+1}$-circuits, and Corollary~\ref{continuous-circuits} tells us that a continuous function $f \colon 2^\NN \rightarrow 2^\NN$ corresponds to a countable collection of finite circuits. Instead of proving the result directly, we will prove a related claim more amenable to induction and rephrased in the language of circuits.

\begin{lemma}[Inductive restriction lemma]
\label{inductive-restriction}

For any ordinal $\alpha < \omega_1$, countable collection of full rank-$\alpha$ alternating trees $(T_{ij})_{i,j \in \NN}$ whose input variable sets are disjoint, and countable collection $(C_n)_{n \in \NN}$ of $\BSigma_\alpha$- or $\BPi_\alpha$-circuits, there exists a restriction $\rho$ such that each $C_n^\rho$ is equivalent to a finite circuit, each $T_{ij}^\rho$ is suppressed or preserved, and infinitely many of the $T_{ij}^\rho$'s are preserved for each fixed $i \in \NN$.

\end{lemma}

Note that each $C_n$ can use any of the variables as inputs. The collection $(T_{ij})_{i,j \in \NN}$ should be interpreted as infinitely many groups of infinitely many alternating trees. The inductive restriction lemma tells us that no matter how these groups are formed, we can always ensure that each group $(T_{ij})_{j \in \NN}$ for fixed $i \in \NN$ contains infinitely many preserved trees.

\subsection{Tree decomposition}

To prove the restriction lemma, we need some information about the combinatorial structure of the full trees, and the lemma we prove in this section is the necessary result. We first recall three important facts about ordinal arithmetic (see \cite{Jech2003SetTheory} for a detailed treatment):

\begin{proposition}[Left subtraction]

Given ordinals $\alpha$ and $\beta$ with $\beta \leq \alpha$, there exists a unique ordinal $\gamma$ such that $\beta + \gamma = \alpha$.

\end{proposition}

\begin{proposition}[Right monotonicity]

Given ordinals $\alpha$, $\beta$ and $\gamma$, we have $\alpha + \beta < \alpha + \gamma$ if and only if $\beta < \gamma$.

\end{proposition}

\begin{proposition}[Right continuity]

Given ordinals $\alpha, \beta_0,\beta_1,\dots$, we have
\begin{equation*}
    \alpha + \limsup_{n \to \infty} \beta_n = \limsup_{n \to \infty}(\alpha + \beta_n) \text{.}
\end{equation*}

\end{proposition}

We also provide some definitions and notation to state our combinatorial lemma. For the following three definitions, $S$ and $T$ are rooted trees. We consider trees up to rooted isomorphism.

\begin{definition}

We write $S \leq T$ if $S$ is a subtree of $T$ with the same root node as $T$.

\end{definition}

\begin{definition}

A \emph{child} of $T$ is a subtree of $T$ rooted at a child of the root node of $T$.

\end{definition}

\begin{definition}

We define the \emph{grafted} tree $S + T$ by replacing the leaves of $T$ by disjoint copies of $S$ (each leaf becomes the root node of a copy of $S$). The root of $S + T$ is the root of $T$.

\end{definition}

A simple transfinite induction verifies that $\rank(S + T) = \rank(S) + \rank(T)$. Our key lemma is the following property of the full trees:

\begin{lemma}
\label{embedding}

We have for all $\alpha,\beta < \omega_1$ that $T_\alpha + T_\beta \leq T_{\alpha+\beta}$.

\end{lemma}

\begin{proof}

We proceed by transfinite induction on $\beta$. If $\beta = 0$, the result is immediate as $T_\alpha + T_0 = T_\alpha$. For the successor step, suppose $\beta = \delta + 1$ and that the claim holds for $\delta$. Then $\alpha + \beta$ is a successor ordinal, so $T_{\alpha + \beta}$ has $\omega$-many $T_{\alpha + \delta}$ trees as children. We also have that $T_{\beta}$ has $\omega$-many $T_{\delta}$ trees as children, so $T_\alpha + T_\beta$ has $\omega$-many $T_\alpha + T_\delta$ trees as children. As $T_\alpha + T_\delta \leq T_{\alpha + \delta}$, this means that $T_\alpha + T_\beta \leq T_{\alpha + \beta}$, proving the successor case.

Now suppose $\beta$ is a limit ordinal and the claim holds for every $\delta < \beta$. Then $T_{\beta}$ has $\omega$-many copies of $T_\delta$ for each $\delta < \beta$ as children. Thus $T_\alpha + T_\beta$ contains as children $\omega$-many copies of the trees $T_\alpha + T_\delta \leq T_{\alpha + \delta}$ for each $\delta < \beta$. Since each $\alpha + \delta < \alpha + \beta$ by right monotonicity, the tree $T_{\alpha + \beta}$ contains $\omega$-many copies of each $T_{\alpha + \delta}$ as children. Thus $T_\alpha + T_\beta \leq T_{\alpha + \beta}$, proving the limit case.
\end{proof}

The above lemmas are important because of the following fact about restrictions, which is immediate from the definitions but crucial in the proof of the restriction lemma:

\begin{lemma}
\label{subtree-restriction}

If $C$ is an alternating circuit obtained from a tree $T$ and $S \leq T$, then there exists a restriction $\rho$ such that $C^\rho$ is equivalent to an alternating circuit obtained from $S$.

\end{lemma}

Finally, we provide one more straightforward lemma which will be helpful in the proof.

\begin{lemma}
\label{killing-leaves}

If an alternating circuit $C$ has rank at most $\beta + \delta$ and $D^\rho$ is equivalent to a finite circuit for every subcircuit $D$ of $C$ with rank at most $\beta$, then $C^\rho$ is equivalent to a circuit with rank at most $\delta$.

\end{lemma}

\subsection{Using the inductive lemma}

We are now equipped with all the tools we need to prove Lemma~\ref{inductive-restriction}. First, let us temporarily assume this lemma to be true and see how it implies Theorem~\ref{restriction} and Corollary~\ref{constant-restriction}

\begin{proof}[Proof of Theorem~\ref{restriction}]

If $\beta = 0$, the claim is immediate. Otherwise, Lemma~\ref{measurable-circuits} tells us that the sets $f^{-1}(\{x \in 2^\NN : x_i = 1\})$ are computed both by $\BSigma_{\alpha+1}$-circuits $C_i$ and $\BPi_{\alpha+1}$-circuits $D_i$ for each $i \in \NN$. Lemma~\ref{embedding} tells us that $T_\alpha + T_\beta \leq T_{\alpha + \beta}$, so Lemma~\ref{subtree-restriction} gives us a restriction $\tilde{\rho}$ such that $g^{\tilde{\rho}}$ is computed by an alternating circuit obtained from the tree $T_\alpha + T_\beta$. Apply Lemma~\ref{inductive-restriction} to every child of the $C_i$'s and $D_i$'s and to the $T_\alpha$ subtrees of $T_\alpha + T_\beta$, where each group consists of all $T_\alpha$ subtrees corresponding to leaves in $T_\beta$ with the same parent. Letting $\rho$ denote this further restriction, we see that each $C_i^{\rho}$ is a $\BSigma_1$-circuit, each $D_i^{\rho}$ is a $\BPi_1$-circuit, and $g^{\rho}$ is a $\beta$-tree function. Lemma~\ref{measurable-circuits} tells us that $f^\rho$ is $\BSigma^0_1$-measurable and hence continuous.
\end{proof}

\begin{proof}[Proof of Corollary~\ref{constant-restriction}]

Define $\tilde{f} \colon 2^\NN \rightarrow 2^\NN$ so that $\tilde{f}(x)_0 = f(x)$ and $\tilde{f}(x)_i = 0$ for all $x \in 2^\NN$ and $i \in \NN \setminus \{0\}$. Theorem~\ref{restriction} then gives us a restriction $\tilde{\rho}$ for which $g^{\tilde{\rho}}$ is computed by a full alternating tree of rank $\beta \geq 1$ and $(\tilde{f})^{\tilde{\rho}}$ is continuous. Corollary~\ref{continuous-circuits} tells us that $f^{\tilde{\rho}}$ depends on only finitely many variables $x_{i_0}, \dots, x_{i_k}$. By fixing each of these variables such that their corresponding input gates are suppressed, we obtain a further restriction $\rho$ for which $f^\rho$ is constant and $g^\rho$ is still a $\beta$-tree function.
\end{proof}

\subsection{Proof of the inductive lemma}

It remains to prove Lemma~\ref{inductive-restriction}, and we proceed by transfinite induction on $\alpha$. We can take $\rho$ to preserve all variables in the $\alpha = 0$ case. We first give a proof of the lemma when $\alpha = 1$, since it will be helpful for both the inductive and limit stages of the induction. The proof when $\alpha = 1$ is similar to Sipser's argument in \cite{sipser1983}.

\begin{proof}[Proof of the $\alpha = 1$ case]

We are given a collection $(T_{ij})_{i,j \in \NN}$ of full rank-$1$ alternating trees with disjoint input variable sets, each of which contains a single $\AND$- or $\OR$-gate with countably infinite fan-in, as well as $\BSigma_1$- or $\BPi_1$-circuits $(C_n)_{n \in \NN}$ defined on all of the input variables. Our goal is to find a restriction $\rho$ such that each $T_{ij}^\rho$ is suppressed or preserved, infinitely many $T_{ij}^\rho$'s are preserved for each fixed $i$, and each $C_n^\rho$ is equivalent to a finite circuit.

The challenge of constructing $\rho$ is that we need to assign infinitely many $*$'s in order to ensure that infinitely many $T_{ij}^\rho$'s are preserved for each fixed $i$, but not in a manner that makes any fixed $C_n^\rho$ depend on infinitely many $*$'s. To address this difficulty, we define the \emph{queue sequence} $(i_n)_{n \in \NN} = (0,0,1,0,1,2,0,1,2,3,\dots)$ so that every natural number appears infinitely many times in the sequence. We define $\rho$ in stages, assigning one $*$ at each stage and ensuring that after the $n$-th stage, the circuit $C_n$ depends only on the $n+1$ variables preserved so far. In particular, we do the following for $n = 0, 1, 2, \dots$:
\begin{enumerate}
    \item Preserve some previously untouched $T_{i_nj}$ by preserving one of its inputs and fixing the remaining inputs.
    \item For each of the $2^{n+1}$ possible settings of the preserved variables to $0$ or $1$, do the following:
    \begin{enumerate}
        \item If $C_n$ is unforced under this setting, it must have an unforced $\BSigma_0$/$\BPi_0$ child. Force this child to be overriding by setting finitely many new variables to $0$ or $1$.
        \item For each of the finitely many trees $T_{ij}$ containing a newly assigned input variable, fix all other input variables of $T_{ij}$ so that $T_{ij}$ is suppressed.
    \end{enumerate}
\end{enumerate}
Let $\tilde{\rho}$ denote the resulting partial restriction. Each $C_n^{\tilde{\rho}}$ depends only on the $n+1$ bits preserved in steps $0$ through $n$ and is hence equivalent to a finite circuit. We also have that infinitely many $T_{ij}$'s are preserved for each fixed $i$. By assigning all remaining literals so that every other $T_{ij}$ is suppressed or preserved, we get our desired restriction $\rho$.
\end{proof}

\begin{proof}[Proof of the successor case]

Let $(T_{ij})_{i,j \in \NN}$ be full rank-$(\alpha+1)$ alternating trees with disjoint input sets and $(C_n)_{n \in \NN}$ be $\BSigma_{\alpha+1}$- or $\BPi_{\alpha+1}$-circuits defined on all inputs. We apply the inductive hypothesis to the collection of all children of the $T_{ij}$'s and to the collection of all $\BSigma_\alpha$- or $\BPi_\alpha$-children of the $C_n$'s. We group the rank-$\alpha$ children of the $T_{ij}$'s by their parent so that each $T_{ij}$ has infinitely many preserved children. We therefore have for the resulting restriction $\tilde{\rho}$ that each $T_{ij}^{\tilde{\rho}}$ is equivalent to a full rank-$1$ alternating tree and each $C_n^{\tilde{\rho}}$ is equivalent to a $\BSigma_1$- or $\BPi_1$-circuit. The $\alpha = 1$ case then gives us a further restriction $\rho$ satisfying the desired properties.
\end{proof}

At this point, we have proven the restriction lemma for finite $\alpha$, which in particular suffices to prove the Borel hierarchy theorem for finite-rank Borel sets. This matches Sipser's result in \cite{sipser1983}. His proof does not involve Lemma~\ref{inductive-restriction}, but rather repeatedly applies the $\alpha = 1$ case of Lemma~\ref{inductive-restriction} to the rank-$1$ subcircuits of an alternating tree $T$ and an arbitrary circuit $C$. This reduces the rank of both $T$ and $C$ by $1$.

This technique breaks down when $\alpha$ is infinite, since replacing all rank-$1$ subcircuits of $C$ with finite circuits need not reduce the complexity of $C$. For this reason, we need a powerful inductive hypothesis (i.e.\ Lemma~\ref{inductive-restriction}) which can reduce the complexity by an infinite amount. Additionally, we need a new procedure for finding restrictions when $\alpha$ is a limit ordinal. Our strategy will be similar in spirit to the $\alpha = 1$ case, but instead of repeatedly reducing $\BSigma_1$- or $\BPi_1$-circuits to finite circuits by setting finitely many literals, we will repeatedly apply the inductive hypothesis to reduce entire $\BSigma_\beta$- or $\BPi_\beta$-circuits (with $\beta < \alpha$) to finite circuits. Each application of the inductive hypothesis must be carefully done, in order to ensure that the $T_{ij}$'s all remain more complex than the $C_n$'s.

\begin{proof}[Proof of the limit case]

Suppose $\alpha$ is a limit ordinal, let $(T_{ij})_{i,j \in \NN}$ be full rank-$\alpha$ alternating trees with disjoint input sets, and suppose $(C_n)_{n \in \NN}$ are $\BSigma_{\alpha}$- or $\BPi_{\alpha}$-circuits defined on all inputs. Let $(D_m)_{m \in \NN}$ denote the children of these circuits. We write $S_{ij0},S_{ij1}, \dots$ to denote the children of the $T_{ij}$'s.

Let $\rho_{-1}$ denote the restriction preserving all variables. We will proceed in stages $s = 0,1,2,\dots$ while building a sequence of restrictions $\rho_{-1} \preceq \rho_0 \preceq \rho_1 \preceq \rho_2 \preceq \cdots$. We can view this process as slowly building a restriction by fixing more and more variables over time. At each stage, we will replace each $D_m$ and $S_{ijk}$ with a new circuit equivalent to the old one under $\rho_s$. We will inductively maintain the invariant that each non-forced $S_{ijk}$ computes a tree function and that
\begin{equation*}
    \rank(D_m) < \limsup_{k \rightarrow \infty} \left(\rank\left(S_{ijk}\right)\right)
\end{equation*}
for all $i,j,m \in \NN$. This invariant holds before any stages are carried out, since the right hand side is equal to $\alpha$. Finally, we will also assign at each stage an index $n_s \in \NN \setminus \{n_0, \dots, n_{s-1}\}$ and call $x_{n_s}$ a \emph{permanently preserved} variable. We will ensure that $\rho_s\left(x_{n_r}\right) = *$ for every $r,s \in \NN$.

Define the queue sequence $(i_s,j_s)_{s \in \NN}$ so that every pair in $\NN^2$ appears infinitely many times in the sequence, and do the following for $s = 0, 1, 2, \dots$:
\begin{enumerate}
    \item Choose a subtree $S_{i_sj_sk}$ which has not yet been forced or preserved and preserve it, noting that only a single variable remains unforced. Set $n_s$ to the index of this variable, so that $x_{n_s}$ is permanently preserved.
    \item Let $\beta_s = \rank(D_s)$, and suppress all subtrees $S_{ijk}$ containing no permanently preserved variables and satisfying $\rank(S_{ijk}) \leq \beta_s$. The invariant ensures that infinitely many $S_{ijk}$ are not suppressed for all $i,j \in \NN$.
    \item For each subtree $S_{ijk}$ with rank greater than $\beta_s$ and containing no permanently preserved variables, let $\gamma_{ijk}$ denote the unique ordinal such that $\beta_s + \gamma_{ijk} = \rank(S_{ijk})$. Using Lemma~\ref{embedding} and Lemma~\ref{subtree-restriction}, we can fix some inputs of $S_{ijk}$ so that $S_{ijk}$ is equivalent to an alternating circuit obtained from the tree $T_{\beta_s} + T_{\gamma_{ijk}}$.
    \item Let $\mathcal{F}$ denote the family of all subcircuits of the $D_m$'s with rank at most $\beta_s$. Note that $D_s \in \mathcal{F}$.
    \item Replace each $D \in \mathcal{F}$ with the equivalent circuit
    \begin{equation*}
        \bigvee_{\sigma_0 \cdots \sigma_s \in \{0,1\}^{s+1}} \left(D^{\sigma_0 \cdots \sigma_s} \wedge \bigwedge_{r=0}^s (x_{n_r} = \sigma_r)\right) \text{,}
    \end{equation*}
    where $D^{\sigma_0 \cdots \sigma_s}$ denotes the circuit $D$ with each permanently preserved literal $x_{n_r}$ set to $\sigma_r$.
    \item Now apply the inductive hypothesis to all of the circuits $D^{\sigma_0 \cdots \sigma_s}$ from the previous step and to the subtrees $T_{\beta_s}$, where each group consists of all $T_{\beta_s}$ subtrees of a given tree $T_{\beta_s} + T_{\gamma_{ijk}}$ corresponding to leaves with the same parent. Let $\rho_s$ denote the restriction so obtained. Because the $D^{\sigma_0 \cdots \sigma_s}$'s and $T_{\beta_s}$'s depend only on preserved variables which are not permanently preserved, we may assume that $\rho_{s-1} \preceq \rho_s$ and that $\rho_s(n_r) = *$ for all $0 \leq r \leq s$.
    \item Replace each $D^{\sigma_0 \cdots \sigma_s}$ and $T_{\beta_s}$ with its equivalent simplified version under $\rho_s$.
\end{enumerate}
We make a few observations about the resulting circuit at the end of each stage. First, since each $\left(D^{\sigma_0 \cdots \sigma_s}\right)^{\rho_s}$ is equivalent to a finite circuit, so too is each circuit
\begin{equation*}
    \bigvee_{\sigma_0 \cdots \sigma_s \in \{0,1\}^{s+1}} \left(D^{\sigma_0 \cdots \sigma_s} \wedge \bigwedge_{r=0}^s x_{n_r} = \sigma_r\right)^{\rho_s}
\end{equation*}
Thus each $(D_m)^{\rho_s}$ with $D_m \in \mathcal{F}$ is equivalent to a finite circuit, and otherwise, we have by Lemma~\ref{killing-leaves} that $(D_m)^{\rho_s}$ is equivalent to a circuit of rank at most $\delta_m$, where $\delta_m$ is the unique ordinal satisfying $\beta_s + \delta_m = \rank(D_m)$. Because $D_s \in \mathcal{F}$, we see that $(D_s)^{\rho_s}$ is equivalent to a finite circuit.

For each $S_{ijk}$ which is not permanently preserved and with rank greater than $\beta_s$, the restricted circuit $\left(S_{ijk}\right)^{\rho_s}$ is equivalent to a full alternating tree of rank $\gamma_{ijk}$, with $\beta_s + \gamma_{ijk} = \rank(S_{ijk})$. For every $i,j,m \in \NN$ such that $\rank(D_m) \geq \beta_s$, we therefore have
\begin{equation*}
    \beta_s + \delta_m = \rank(D_m) < \limsup_{k \rightarrow \infty} \left(\rank(S_{ijk})\right) = \beta_s + \limsup_{k \rightarrow \infty} \gamma_{ijk}
\end{equation*}
by right continuity. This means that $\delta_m < \limsup_{k \rightarrow \infty} \gamma_{ijk}$ by right monotonicity, so the invariant is maintained after each step (when $\rank(D_m) < \beta_s$, the invariant reduces to the fact that $0 < \limsup_{k \rightarrow \infty} \gamma_{ijk}$).

Let $\tilde{\rho}$ denote the limiting restriction, refined so that any remaining $S_{ijk}$'s are suppressed or preserved. We then have that each $(T_{ij})^{\tilde{\rho}}$ is equivalent to a full rank-$1$ tree. Since each $(D_m)^{\tilde{\rho}}$ is equivalent to a finite circuit, each $(C_n)^{\tilde{\rho}}$ is equivalent to a $\BSigma_1$- or $\BPi_1$-circuit. Thus we can obtain our desired restriction $\rho$ from the $\alpha = 1$ case to complete the induction in the limit case and prove the theorem.
\end{proof}

\section{Consequences of the restriction lemma}
\label{sec:consequences}

Along with being interesting in its own right, the restriction lemma is a powerful tool for proving lower bounds on Borel rank. We demonstrate this fact by providing new proofs for several classical results in descriptive set theory.

\subsection{The Borel hierarchy theorem}
\label{sec:hierarchy}

\begin{theorem}[Borel hierarchy]
\label{borel-hierarchy}

We have for every $1 \leq \gamma < \omega_1$ that $\BSigma^0_\gamma \neq \BPi^0_\gamma$.

\end{theorem}

This theorem is a cornerstone result of descriptive set theory which, along with reductions to complete sets, is a key piece of machinery for proving lower bounds on Borel rank. The usual proof requires diagonalizing against universal sets constructed for each class $\BSigma^0_\gamma$ and $\BPi^0_\gamma$. The restriction lemma gives us a purely combinatorial proof:

\begin{proof}

It suffices to prove the claim when $\gamma$ is a successor ordinal, so write $\gamma = \alpha + 1$. Let $g \colon 2^\NN \rightarrow \{0,1\}$ be a $\gamma$-tree function. Since $g^{-1}(\{1\}) \in \BSigma^0_\gamma \cup \BPi^0_\gamma$, it suffices to show that $g^{-1}(\{1\})$ is not in $\BDelta^0_\gamma = \BDelta^0_{\alpha+1}$. Indeed, if this were the case, then $g$ would be $\BSigma^0_{\alpha+1}$-measurable. Applying Corollary~\ref{constant-restriction} with $\beta = 1$ would then give us a restriction $\rho$ with $g^\rho$ simultaneously constant and also a $1$-tree function, which is a contradiction.
\end{proof}

\begin{remark}

This result also implies the Borel hierarchy theorem in any uncountable Polish space $X$, since any such space contains a homeomorphic copy $K$ of the Cantor space. The image of a set which is $\BSigma^0_\gamma \setminus \BPi^0_\gamma$ (for $\gamma \geq 2$) in $2^\NN$ under a homeomorphism $\varphi \colon 2^\NN \rightarrow K$ can easily be shown to be $\BSigma^0_\gamma \setminus \BPi^0_\gamma$ in $X$.

\end{remark}

\subsection{Ramsey-theoretic results}
\label{sec:constant}

The restriction lemma allows us to prove some Ramsey-like properties of Borel functions.

\begin{theorem}
\label{continuous-subcube}

For any Borel function $f \colon 2^\NN \rightarrow 2^\NN$, there exists a restriction $\rho$ on $2^\NN$ preserving infinitely many variables such that $f^\rho$ is continuous.

\end{theorem}

\begin{proof}

Fix a Borel function $f$, and choose $\alpha < \omega_1$ such that $f$ is $\BSigma^0_{\alpha+1}$-measurable. Let $g \colon 2^\NN \rightarrow \{0,1\}$ be an $(\alpha+1)$-tree function. Theorem~\ref{restriction} (with $\beta = 1$) gives us a restriction $\rho$ such that $f^\rho$ is continuous while $g^\rho$ is a $1$-tree function. As $g^\rho$ depends on infinitely many variables, $\rho$ must preserve infinitely many variables.
\end{proof}

\begin{remark}

The domain of any restricted function is a subcube of the Cantor space. Thus the above result can be interpreted as saying that, if $f \colon 2^\NN \rightarrow 2^\NN$ is Borel, then $f|_C$ is continuous for some infinite subcube $C \subseteq 2^\NN$.

\end{remark}

\begin{remark}

There are other proofs of Theorem~\ref{continuous-subcube}. One such approach is to use the facts that any Borel function $f \colon 2^\NN \rightarrow 2^\NN$ can be made continuous by restricting to a comeager subset of $2^\NN$ and that every comeager subset of $2^\NN$ contains an infinite subcube. A simple argument involving the Galvin--Prikry theorem \cite{galvin1973borel} also gives us Theorem~\ref{continuous-subcube}. Both of these alternative approaches, however, rely on either topological machinery or heavy theory.

\end{remark}

We also obtain the following Ramsey-like result from the constant restriction lemma:

\begin{theorem}
\label{constant-subcube}

For any Borel function $f \colon 2^\NN \rightarrow \{0,1\}$, there exists a restriction $\rho$ on $2^\NN$ preserving infinitely many variables such that $f^\rho$ is constant.

\end{theorem}

\begin{proof}

Choose $\alpha < \omega_1$ such that $f$ is $\BSigma^0_{\alpha+1}$-measurable, and let $g \colon 2^\NN \rightarrow \{0,1\}$ be an $(\alpha+1)$-tree function. Use Corollary~\ref{constant-restriction} (with $\beta = 1$) to obtain a restriction $\rho$ such that $f^\rho$ is constant while $g^\rho$ is a $1$-tree function, which implies that $\rho$ preserves infinitely many variables.
\end{proof}

\begin{remark}

These results are much weaker than the restriction lemma, since they ``forget'' the precise structure of $\rho$ and its relation to $g$. In each proof, we only use the fact that $\rho$ preserves infinitely many variables.

\end{remark}

\subsection{Infinite parity}
\label{sec:parity}

An \emph{infinite parity function} is a function $f \colon 2^\NN \rightarrow \{0,1\}$ such that $f(x) \neq f(y)$ whenever $x_i \neq y_i$ for exactly one $i \in \NN$. It is a classical result in descriptive set theory that no infinite parity function is Borel. Our argument is perhaps the most similar to Furst, Saxe, and Sipser's proof of the finitary analog: that parity is not in $\AC$ \cite{parity1981}.

\begin{corollary}

No infinite parity function is Borel.

\end{corollary}

\begin{proof}

Fix a Borel function $f$, and use Theorem~\ref{constant-subcube} to obtain a restriction $\rho$ preserving infinitely many variables for which $f^\rho$ is constant. For any infinite parity function $g$, the restricted function $g^\rho$ is another infinite parity function, which is not constant. Thus $f^\rho \neq g^\rho$, which implies that $f \neq g$.
\end{proof}

\section{Generalization to non-trees}
\label{sec:restriction-for-non-trees}

Theorem~\ref{restriction} tells us that, if $f \colon 2^\NN \rightarrow 2^\NN$ is $\BSigma_{\alpha + 1}$-measurable and $A \subseteq 2^\NN$ is a $\BSigma^0_{\alpha + \beta}$- or $\BPi^0_{\alpha + \beta}$-complete set computed by a tree function, then there is a subcube $C \subseteq 2^\NN$ such that $f|_C$ is continuous and $A \cap C$ is $\BSigma^0_\beta$- or $\BPi^0_\beta$-hard. It is a natural question whether we can take $A$ to be any $\BSigma^0_{\alpha + \beta}$- or $\BPi^0_{\alpha + \beta}$-complete set, not necessarily computed by a tree function. The answer is yes if we allow for $C$ to be any closed set rather than just a subcube:

\begin{theorem}
\label{continuous-non-trees}

Fix ordinals $\alpha, \beta < \omega_1$. Then for any $\BSigma^0_{\alpha + 1}$-measurable function $f \colon 2^\NN \rightarrow 2^\NN$ and $\BSigma^0_{\alpha + \beta}$- or $\BPi^0_{\alpha + \beta}$-complete set $A$, there exists a closed set $C \subseteq 2^\NN$ such that $f|_C$ is continuous and $A \cap C$ is $\BSigma^0_\beta$- or $\BPi^0_\beta$-hard.

\end{theorem}

\begin{proof}

Fix an $(\alpha + \beta)$-tree function $g \colon 2^\NN \rightarrow \{0,1\}$ computing a set in the appropriate class $\BSigma^0_{\alpha + \beta}$ or $\BPi^0_{\alpha + \beta}$ so that there exists a continuous reduction $\varphi \colon 2^\NN \rightarrow 2^\NN$ with $g(x) = 1 \iff \varphi(x) \in A$. The function $f \circ \varphi$ is $\BSigma^0_{\alpha+1}$-measurable, so Theorem~\ref{restriction} gives us a restriction $\rho$ such that $(f \circ \varphi)^\rho$ is continuous and $g^\rho$ is computed by a $\beta$-tree function. Let $C'$ denote the subcube defined by $\rho$, and set $C = \varphi(C')$. Then $C$ is compact and hence closed. Because $(f \circ \varphi)|_{C'}$ is continuous, $\varphi$ is continuous, and $C'$ is compact, a standard topological argument tells us that $f|_C$ is continuous (or one can simply note that $\varphi$ is a quotient map). Finally, since $\varphi|_{C'}$ is a reduction from the $\BSigma_\beta^0$- or $\BPi_\beta^0$-hard set computed by $g^\rho$ to $A \cap C$, the set $A \cap C$ is $\BSigma_\beta^0$- or $\BPi_\beta^0$-hard.
\end{proof}

Applying a similar argument using Corollary~\ref{constant-restriction} gives us the following weaker result:

\begin{theorem}
\label{constant-non-trees}

Fix ordinals $\alpha, \beta < \omega_1$ such that $\beta \geq 1$. Then for any $\BSigma^0_{\alpha + \beta}$- or $\BPi^0_{\alpha + \beta}$-complete set $A$ and set $B \in \BDelta^0_{\alpha+1}$, there exists a closed set $C \subseteq 2^\NN$ such that $A \cap C$ is $\BSigma^0_\beta$- or $\BPi^0_\beta$-hard and $B \cap C$ is either $\emptyset$ or $C$.

\end{theorem}

\begin{remark}

One can also prove these results classically using a refinement of topology.

\end{remark}

\subsection{Limitations in the infinite setting}

The following counterexample tells us that this relaxation from subcubes to closed sets is necessary when considering $\BSigma^0_\gamma$- or $\BPi^0_\gamma$-complete sets which are not computed by tree functions, even for the weaker result in Theorem~\ref{constant-non-trees}.

\begin{theorem}
\label{counterexample}

For any countable ordinal $\gamma$, there exist Borel functions $f,g \colon 2^\NN \rightarrow \{0,1\}$ such that $f^{-1}(\{1\})$ is open, $g$ computes a $\BSigma^0_\gamma$- or $\BPi^0_\gamma$-hard set, and for any restriction $\rho$ such that $f^\rho$ is constant, the function $g^\rho$ is also constant.

\end{theorem}

\begin{proof}

We implement the parity gate $\oplus$ on two bits using the formula $x_0 \oplus x_1 = (\overline{x}_0 \AND x_1) \OR (x_0 \AND \overline{x}_1)$ and define the functions $d,h \colon 2^\NN \rightarrow 2^\NN$ by $ d(x) = x_0x_0x_1x_1x_2x_2 \dots$ and $h(x) = x_0x_2x_4 \dots$. Let $\tau \colon 2^\NN \rightarrow \{0,1\}$ be a $\gamma$-tree function. We then set
\begin{equation*}
    f(x) = \bigvee_{i \in \NN} x_{2i} \oplus x_{2i + 1} \quad \text{and} \quad g(x) = f(x) \vee \tau(h(x)) \text{.}
\end{equation*}
The function $f$ is computed by a $\BSigma_1$-circuit, so $f^{-1}(\{1\})$ is open. We also have that $\tau = g \circ d$, and since $d$ is a continuous reduction from the $\BSigma^0_\gamma$- or $\BPi^0_\gamma$-complete set computed by $\tau$ to $g^{-1}(\{1\})$, this means that  $g$ computes a $\BSigma^0_\gamma$- or $\BPi^0_\gamma$-hard set. Finally, suppose $\rho$ is a restriction for which $f^\rho$ is constant. If $f^\rho$ is constantly $1$, then $g^\rho$ is constantly $1$. If $f^\rho$ is constantly $0$, then $\rho$ must fix every $x_i$, implying that $g^\rho$ is constant.
\end{proof}

We can also set $g(x) = f(x) \vee p(h(x))$ for some infinite parity function $p$ to obtain the following related counterexample:

\begin{theorem}
\label{parity-counterexample}

There exists a function $f \colon 2^\NN \rightarrow \{0,1\}$ computing an open set and function $g \colon 2^\NN \rightarrow \{0,1\}$ which is not Borel such that, for any restriction $\rho$ such that $f^\rho$ is constant, the function $g^\rho$ is also constant.

\end{theorem}

\begin{remark}

Continuous reductions allow us to circumvent the difficulty of working with sets not computed by tree functions, but they destroy the combinatorial structure of restrictions. These counterexamples tell us that this is a necessary trade-off.

\end{remark}

\subsection{Limitations for constant-depth circuit families}
\label{sec:ac-limitations}

Let $\ACd$ denote the class of problems computable by depth-$d$ polynomial-sized circuit families $(C_n)_{n \in \NN}$, and set $\AC = \bigcup_{d \in \NN} \ACd$. The proof of Theorem~\ref{counterexample} carries over seamlessly to this setting, and using Sipser's result in \cite{sipser1983} that the finite analogs of the infinite tree functions of depth $d$ are not in $\ACdm$ gives us the following result:

\begin{theorem}

For any $d \geq 2$, there exist functions $f_n,g_n \colon \{0,1\}^n \rightarrow \{0,1\}$ for $n \in \NN$ such that $\{f_n\}_{n \in \NN} \in \mathsf{AC}^0_2$, $\{g_n\}_{n \in \NN} \in \ACd \setminus \mathsf{AC}^0_{d - 1}$, and for every $n \in \NN$ and restriction $\rho$ such that $f_n^\rho$ is constant, the function $g_n^\rho$ is also constant.

\end{theorem}

Following Theorem~\ref{parity-counterexample} instead, we get the following:

\begin{theorem}

There exist functions $f_n,g_n \colon \{0,1\}^n \rightarrow \{0,1\}$ for $n \in \NN$ such that $\{f_n\}_{n \in \NN} \in \mathsf{AC}^0_2$, $\{g_n\}_{n \in \NN} \notin \AC$, and for every $n \in \NN$ and restriction $\rho$ such that $f_n^\rho$ is constant, the function $g_n^\rho$ is also constant.

\end{theorem}

Restrictions are certainly a powerful tool for proving the limitations of $\AC$. Just as in the infinite case, though, our counterexamples show that they are not universally applicable for proving lower bounds on arbitrary functions.

\section{Sharpness of the restriction lemma}
\label{sec:sharpness}

When applying the restriction lemma to reduce the complexity of a function $f$ by $\alpha$, we also reduce the complexity of some $(\alpha+\beta)$-tree function $g$ by $\alpha$. A natural question is whether we can reduce the complexity of $f$ without reducing the complexity of $g$, and it turns out that the answer depends on whether the codomain of $f$ is $2^\NN$ or $\{0,1\}$. The first case corresponds to the setting of Theorem~\ref{restriction}, and the result is sharp in this scenario. When $f$ maps into $\{0,1\}$, on the other hand, we are in the setting of Corollary~\ref{constant-restriction}. In this scenario, the result is not sharp.

\subsection{Optimality of Theorem~\ref{restriction}}

The following result tells us that Theorem~\ref{restriction} is completely sharp.

\begin{theorem}

For any ordinals $\alpha,\beta < \omega_1$ and $(\alpha+\beta)$-tree function $g \colon 2^\NN \rightarrow \{0,1\}$, there is a $\BSigma^0_{\alpha+1}$-measurable function $f \colon 2^\NN \rightarrow 2^\NN$ such that, for any restriction $\rho$ such that $f^\rho$ is continuous, the function $g^\rho$ computes a Borel set of rank at most $\beta$.

\end{theorem}

\begin{proof}

Let $C$ be an alternating tree of rank $\alpha+\beta$, and let $\mathcal{F} = \{C_0,C_1,\dots\}$ denote the family of subcircuits of $C$ with rank at most $\alpha$. Define $f \colon 2^\NN \rightarrow 2^\NN$ such that $f(x)_i = C_i(x)$. Lemma~\ref{measurable-circuits} tells us that $f$ is $\BSigma^0_{\alpha+1}$-measurable. If $\rho$ is a restriction such that $f^\rho$ is continuous, then Corollary~\ref{continuous-circuits} tells us that each $C_i^\rho$ is equivalent to a finite circuit. Thus for this restriction $\rho$, the function $g^\rho$ is computed by a $\BSigma_\beta$- or $\BPi_\beta$-circuit. Thus $g^\rho$ computes a set of rank at most $\beta$.
\end{proof}

\begin{remark}

Theorem~\ref{restriction} tells us that $\rho$ may always be chosen so that $f^\rho$ is continuous and $g^\rho$ computes a $\BSigma^0_\beta$- or $\BPi^0_\beta$-complete set, so both results are optimal for certain choices of $f$.

\end{remark}

\subsection{Non-optimality of Corollary~\ref{constant-restriction}}

While Theorem~\ref{restriction} can be viewed as reducing the rank of countably many circuits (each coordinate of $f$ is computed by a $\BSigma_{\alpha+1}$- and $\BPi_{\alpha+1}$-circuit), Corollary~\ref{constant-restriction} only reduces the rank of a single circuit. In this way, Corollary~\ref{constant-restriction} is weaker than Theorem~\ref{restriction}, and it turns out that Corollary~\ref{constant-restriction} is not even sharp. For example, Corollary~\ref{constant-restriction} tells us that for any $\BSigma_1$-circuit $C$ and $2$-tree function $g$, there exists a restriction $\rho$ such that $g^\rho$ is a $1$-tree function and $C^\rho$ is constant. However, we can actually choose $\rho$ such that $g^\rho$ is a $2$-tree function and $C^\rho$ is constant. This is because one can always force $C$ to be constant by setting finitely many bits, and $\rho$ can be obtained by then suppressing finitely many rank-$1$ subcircuits in a full alternating rank-$2$ tree computing $g$. Using this observation in conjunction with Lemma~\ref{inductive-restriction}, one can obtain a similar rank-$1$ improvement over Corollary~\ref{constant-restriction} for larger values of $\alpha$ and $\beta$.

It is therefore clear that Corollary~\ref{constant-restriction} is not sharp. In fact, it is an open problem whether forcing $C^\rho$ to be constant requires the complexity of an appropriate tree function (with rank larger than that of $C$) to drop at all. We pose this question more precisely as follows:

\begin{conjecture}
\label{rank-preservation}

Suppose $\alpha < \omega_1$, a circuit $C$ has rank strictly less than $\alpha$, and $g$ is an $\alpha$-tree function. Then there exists a restriction $\rho$ such that $C^\rho$ is constant and $g^\rho$ is still an $\alpha$-tree function.

\end{conjecture}

There are a few reasons to suspect that the conjecture holds, and we present two partial results in this direction. The first is that such a restriction can always be found when $\alpha \leq 3$. We prove this via an optimized version of the construction we used to prove the $\alpha = 1$ case of Lemma~\ref{inductive-restriction}.

\begin{theorem}

Conjecture \ref{rank-preservation} holds whenever $\alpha \leq 3$.

\end{theorem}

\begin{proof}

The $\alpha = 1$ case is trivial, and we already addressed the $\alpha = 2$ case. It therefore suffices to prove that, whenever $C$ is a $\BPi_2$-circuit and $g$ is a $3$-tree function computing a set in $\BSigma^0_3$, there exists a restriction $\rho$ for which $C^\rho$ is constant and $g^\rho$ is a $3$-tree function (the other cases are analogous).

To this end, relabel the input variables $x_0,x_1,\dots$ so that $g(x) = \bigvee_{i \in \NN} \bigwedge_{j \in \NN} \bigvee_{k \in \NN} x_{ijk}$. We call the collection $\{x_{ijk} : k \in \NN\}$ of variables the \emph{$(i,j)$-th row}. Note that any partial restriction setting variables in only finitely many rows can be extended to a restriction $\rho$ for which $g^\rho$ is a $3$-tree function. Call a row \emph{starred} if it contains at least one preserved variable. Write $C$ as the conjunction of $\BSigma_1$-circuits $D_n$ for $n \in \NN$. Define the \emph{queue sequence} $(i_n,\ell_n)_{n \in \NN}$ so that every $(i,\ell) \in \NN^2$ appears infinitely many times in the sequence.

We are done if there exists a restriction $\rho$ such that $g^\rho$ is a $3$-tree function and some $D_n^\rho$ is identically $0$, so suppose this is not the case. We will construct a restriction $\rho$ such that $g^\rho$ is a $3$-tree function and every $D_n^\rho$ is identically $1$. We define $\rho$ in stages, assigning one $*$ at each stage and ensuring that after the $n$-th stage, the circuit $D_n$ is forced to $1$. As we construct $\rho$, we also define \emph{label functions} $J_i \colon \NN \rightarrow \NN$. Each $J_i$ will be an injective partial function at finite stages of the construction and a total injective function in the limit. We will ensure that at the end of each stage, the $(i,j)$-th row is starred if and only if $j$ lies in the image of $J_i$. In particular, we do the following for $n = 0, 1, 2, \dots$:
\begin{enumerate}
    \item If $J_{i_n}(\ell_n)$ is not yet defined, choose some $j \in \NN$ not in the range of $J_{i_n}$ such that $\rho$ has not yet been defined on any variables in the $(i_n,j)$-th row. Define $J_{i_n}(\ell_n) = j$.
    \item Preserve an additional variable in the $(i_n,J_{i_n}(\ell_n))$-th row.
    \item For each of the $2^{n+1}$ possible settings of the preserved variables to $0$ or $1$, do the following:
    \begin{enumerate}
        \item Under this setting, it must be possible to force $D_n$ to $1$ while assigning only $0$'s in starred rows (otherwise, setting all remaining variables in the starred rows to $0$ would force $D_n$ to $0$, and we could extend this partial restriction to some $\rho$ with $g^\rho$ being a $3$-tree function and $D_n^\rho$ being identically $0$). Thus one of the $\BSigma_0$/$\BPi_0$-children of $D_n$ can be forced to $1$ by setting finitely many variables, assigning only $0$'s in starred rows.
        \item For each of the finitely many unstarred rows containing a newly assigned variable, assign an additional $1$ in that row so that the row is suppressed.
    \end{enumerate}
\end{enumerate}
Let $\tilde{\rho}$ denote the resulting partial restriction, and note that $\tilde{\rho}$ forces every $D_n^{\tilde{\rho}}$ to $1$. We also see that every starred row contains only $0$'s and $*$'s, and that each starred row contains infinitely many stars. For each fixed $i$, the $(i,j)$-th row is starred for infinitely many $j$, and this means that by extending $\tilde{\rho}$ to a restriction $\rho$ suppressing all other rows, the restricted function $g^\rho$ is a $3$-tree function.
\end{proof}

The second partial result is that the claim holds when the rigid subcubes induced by restrictions are replaced by closed sets (this is the same relaxation we made in Section~\ref{sec:restriction-for-non-trees}). The argument was found by Edward Hou, and it relies heavily on a Hurewicz-type result proven by Louveau and Saint Raymond in \cite{louveau1987}.

\begin{theorem}

Fix $\alpha < \omega_1$ and $A,B \subseteq 2^\NN$ such that $A$ is $\BPi^0_\alpha$-complete and $\rank(B) < \alpha$. Then there exists a closed set $C \subseteq 2^\NN$ such that $A \cap C$ is $\BPi^0_\alpha$-hard and $B \cap C$ is equal to either $\emptyset$ or $C$.

\end{theorem}

\begin{proof}

Because $A \notin \BSigma^0_\alpha$, we must have that either $A \cap B \notin \BSigma^0_\alpha$ or $A \setminus B \notin \BSigma^0_\alpha$. We consider each of these cases separately.

We first suppose that $A \cap B \notin \BSigma^0_\alpha$. We can assume that $\alpha \geq 3$ since the $\alpha \leq 2$ cases are trivial, and this allows us to apply Louveau and Saint Raymond's theorem in \cite{louveau1987} to the sets $A \cap B$ and $B \setminus A$. The theorem tells us that either there exists a $\BSigma^0_\alpha$-set $S$ with $A \cap B \subseteq S \subseteq A \cup B^c$ or there is a continuous map $\varphi \colon 2^\NN \rightarrow 2^\NN$ such that $\varphi(H_\alpha) \subseteq A \cap B$ and $\varphi(H_\alpha^c) \subseteq B \setminus A$, where $H_\alpha$ is some $\BPi^0_\alpha$-complete subset of $2^\NN$. The first alternative cannot hold, since it would imply that $A \cap B = S \cap B \in \BSigma^0_\alpha$. Thus the second alternative holds, and $C = \varphi(2^\NN)$ is closed and entirely contained in $B$. Because $\varphi$ is a continuous reduction from $H_\alpha$ to $A \cap C$, the set $A \cap C$ is $\BPi^0_\alpha$-hard.

Now suppose that $A \setminus B \notin \BSigma^0_\alpha$. Louveau and Saint Raymond's theorem applied to $A \setminus B$ and $(A \cup B)^c$ tells us that either there exists a $\BSigma^0_\alpha$-set $S$ with $A \setminus B \subseteq S \subseteq A \cup B$ or there is a continuous map $\varphi \colon 2^\NN \rightarrow 2^\NN$ such that $\varphi(H_\alpha) \subseteq (A \setminus B)$ and $\varphi(H_\alpha^c) \subseteq (A \cup B)^c$. The first alternative again cannot hold, since we would then have that $A \setminus B = S \setminus B \in \BSigma^0_\alpha$. Thus the second alternative holds, and we see that $C = \varphi(2^\NN)$ is closed and disjoint from $B$. Because $\varphi$ is a continuous reduction from $H_\alpha$ to $A \cap C$, the set $A \cap C$ is again $\BPi^0_\alpha$-hard.
\end{proof}

\begin{remark}

This argument tells us that, if Conjecture \ref{rank-preservation} fails, it must be due to the precise combinatorial structure of restrictions (as was the case for Theorem~\ref{constant-non-trees}).

\end{remark}

\section{Acknowledgments}

I am deeply grateful to Anton Bernshteyn for his mentorship and guidance on this project. He introduced me to descriptive set theory and its connection to circuit complexity, suggested theorems and problems I should look into, and helped me refine my ideas. His insights and support were invaluable.

\printbibliography

\end{document}